\documentclass[letterpaper]{amsart}
\usepackage[utf8]{inputenc}
\usepackage[T1]{fontenc}

\usepackage{times}

\usepackage[pagebackref,colorlinks=true,pdfpagemode=UseNone,urlcolor=blue,
linkcolor=blue,citecolor=blue]{hyperref}

\usepackage{amsmath,amsfonts,amssymb,amsthm}
\usepackage{amsthm, amssymb, amsmath, amsfonts, mathrsfs}

\usepackage{mathrsfs}
\usepackage{MnSymbol}
\usepackage{scalerel} % for scaling the cube

\usepackage{color}
\usepackage{accents}

\usepackage{bbm}

\usepackage[normalem]{ulem}

\definecolor{labelkey}{gray}{.8}
\definecolor{refkey}{gray}{.8}

\definecolor{darkred}{rgb}{0.9,0.1,0.1}
\definecolor{darkgreen}{rgb}{0,0.5,0}

\newcommand{\tnorm}[1]{{\left\vert\kern-0.25ex\left\vert\kern-0.25ex\left\vert #1 
    \right\vert\kern-0.25ex\right\vert\kern-0.25ex\right\vert}}

\newtheorem{theorem}{Theorem}[section]
\newtheorem{lemma}[theorem]{Lemma}

\newtheorem{proposition}[theorem]{Proposition}

\theoremstyle{remark}
\newtheorem{remark}[theorem]{Remark}

\renewenvironment{proof}[1][Proof]{ {\itshape \noindent {#1.}} }{$\Box$
\medskip}

\numberwithin{equation}{section}
\newcommand{\R}{\mathbb{R}}

\newcommand{\Pb}{\mathbb{P}}

\newcommand{\E}{\mathbb{E}}

\newcommand{\D}{\mathcal{D}}

\newcommand{\eps}{\varepsilon}
\newcommand{\ep}{\epsilon}

\newcommand{\VVar}{\mathrm{\VVar}}

\newcommand{\la}{\langle}
\newcommand{\ra}{\rangle}

\newcommand{\EE}{\mathbf{E}}

\renewcommand{\H}{\mathcal{H}}

\def\blue{\textcolor{black}}

\newcommand{\x}{\mathbf{x}}
\newcommand{\y}{\mathbf{y}}

\newcommand{\cZ}{\mathcal{Z}}
\newcommand{\cD}{\mathcal{D}}

\newcommand{\bfX}{\mathbf{X}}

\newcommand{\bfU}{\mathbf{U}}

\newcommand{\pp}{\mathbf{p}}
\newcommand{\LL}{\mathbf{h}_0}
\newcommand{\VV}{\mathbf{R}}
\newcommand{\rhopoly}{\rho_{t,x}^{\scriptscriptstyle\rm poly}}

\newcommand{\Epoly}{\E^{\scriptscriptstyle\rm poly}}
\newcommand{\Egibbs}{\E^{\scriptscriptstyle\rm Gibbs}}
\newcommand{\Einit}{\E^{\scriptscriptstyle\rm init}}
\newcommand{\Ppoly}{{\mathbb P}^{\scriptscriptstyle\rm poly}}

\begin{document}
\title{Integration by parts and invariant measure for KPZ}

\author{Yu Gu and Jeremy Quastel}

%\address[Alexander Dunlap]{Department of Mathematics, Duke University, Durham, NC 27708, USA.}
%\email{dunlap@math.duke.edu}

\address[Yu Gu]{Department of Mathematics, University of Maryland, College Park, MD 20742, USA. }
\email{yugull05@gmail.com}

\address[Jeremy Quastel]{Department of Mathematics, University of Toronto, 40 St. George St, Toronto, ON M5S 2E4, Canada}
\email{quastel@math.toronto.edu}

\maketitle

\begin{abstract} 
Using Stein's method and a Gaussian integration by parts, we provide a direct proof of the known fact that drifted Brownian motions are invariant measures (modulo height) for the KPZ equation.
\medskip

\end{abstract}
\maketitle

%\tableofcontents

\section{Background}

 The  Kardar-Parisi-Zhang (KPZ) equation is,
\begin{equation}\label{KPZ}
\partial_t h = \tfrac12(\partial_x h)^2 +\tfrac12\partial_x^2 h + \xi,
\end{equation}
where $\xi$ denotes space-time white noise, the distribution valued Gaussian field with correlation function
 \begin{equation}
 \langle \xi(t,x),\xi(s,y)\rangle=\delta(t-s)\delta(x-y).
 \end{equation}
It is an equation for a randomly evolving height function $h\in \mathbb{R}$ which depends on position $x\in \mathbb{R}$ and time $t\in \mathbb{R}_+$. The derivative 
\begin{equation}\label{you}
    u=\partial_x h
\end{equation}
solves the (essentially equivalent) stochastic Burgers equation
\begin{equation}
    \label{SBE}
    \partial_t u = \tfrac12\partial_x u^2 +\tfrac12\partial_x^2 u + \partial_x\xi.
\end{equation}%\gucomment{I'm curious why you think Burgers and KPZ are ``equivalent''. you got $t^{2/3}$ from Burgers, but do you know how to get $t^{1/3}$?}\adcomment{Sure.  Current across the origin.}
Although very sophisticated solutions theories for \eqref{KPZ} and \eqref{SBE} have become available in recent years \cite{HairerKPZ,Goncalves2014,Gubinelli2017}, they all identify the solution $h$ as the \emph{Cole-Hopf solution} \begin{equation}\label{aitch} h=\log Z,\end{equation}
where $Z$ is the solution of the stochastic heat equation  
\begin{equation}\label{she}
\partial_t Z=\tfrac12\partial_x^2 Z+Z\xi.
\end{equation}
The stochastic heat equation is one of the few stochastic partial differential equations which is well-posed by an elementary extension of the It\^o theory  \cite{walsh1986introduction}.

One of the amazing facts about KPZ \eqref{KPZ} is the (almost) invariance of Brownian motion.  More precisely, it is invariant modulo height shifts, or, equivalently, its derivative,  white noise, is invariant for SBE \eqref{SBE}.  This invariant spatial white noise  is distinct from the forcing space-time white noise $\xi$, and one should think of it as independent of $\xi$ and living on an orthogonal probability space.  

If one drops the non-linearity in \eqref{KPZ}/\eqref{SBE}, it gives the so-called Edwards-Wilkinson equation, which is an infinite dimensional Ornstein–Uhlenbeck process. The same Brownian motion/white noise is easily seen to be invariant under the linear dynamics.  The fact that the non-linear part of the dynamics also preserves the Brownian motion is not at all obvious, and unfortunately we still do not have a clear simple proof of this fact. In general we are completely lacking methods to compute or even check invariant measures for non-linear stochastic partial differential equations. Part of the problem is the difficulty of constructing domains for the generator $L$ of the semigroup, so that the equation for invariant measures  $\int Lf d\mu=0$ can somehow be checked for those $f$ living in the domain.  In the KPZ case, invariance has previously always been proven by approximating the process by discrete models, or mollifications, for which the invariance of appropriate versions of Brownian motion are known.  And all these arguments involve a properly interpreted telescoping series.    We survey them now.

\subsection{Formal argument \cite{Parisi,Spohn}}  Since the Edwards-Wilkinson equation  $\partial_t u= \tfrac12\partial_x^2 u + \partial_x\xi$  preserves white noise, it is (formally) enough to prove  the invariance for the 
Burgers flow\footnote{The essential obstacle here that the Burgers flow  is ill-defined.
 If one interprets it via entropy solutions,  then one has  the Lax-Oleinik variational formula for the solution, $u=\partial_x h$ with
\begin{equation}
h(t,x) = \sup_{y\in \mathbb{R}} \left\{ -\frac{(x-y)^2}{2t} +h(0,y)\right\}.
\end{equation}
Starting from $h(0,x)$ a two-sided Brownian motion, one obtains a collection of  `N-waves', i.e.  the indefinite space integral of a bunch of Dirac masses of various sizes, minus a linear function.   The statistics are
 known exactly \cite{MR1781884} (see also \cite{MR1833701} for more general classes of solvable initial data).  At any rate, the result starting with a Brownian motion is definitely {\it not} a new Brownian motion, though the formal argument tells you it should be.  Note also that the formal argument works just as well for the ill-defined flows $\partial_tu= \tfrac12\partial_x(u^m)$.} $\partial_tu= \tfrac12\partial_x(u^2)$.  For a nice function $f$  on the state space, we hope to show that under the Burgers flow  $\Einit[\EE[f(u(t))| u_0]]= \Einit[f(u_0)]$ where $\Einit$ refers to the expectation over the white noise initial data, and $\EE[\cdot | u_0]$ over the process starting from $u_0$.  We write it formally as 
\begin{equation}\label{aaa}
\partial_t \int du~ P_t f(u) e^{-\tfrac{\sigma^2}{2}\int u^2}=0.
\end{equation} 
The (ill-defined) integral is over the state space with (mythical) flat measure $du$.  The measure $e^{-\tfrac{\sigma^2}{2}\int u^2}du$ indicates  white noise with 
$
E[ u(x) u(y)] = \sigma^{-2} \delta(x-y)
$.  $P_t$ denotes the semigroup $P_tf(u)=\EE[f(u(t))\mid u(0)=u]$. The formal generator of the process is 
\begin{equation}
    \mathcal{L}f = \int dx~ \tfrac12\partial_x(u^2) \frac{\delta f}{\delta u_x},
\end{equation}
where $ \frac{\delta f}{\delta u}$ is the functional (Frechet) derivative.  Differentiating the left hand side of \eqref{aaa} gives
\begin{equation}
  \int du \int dx  \frac{\delta f}{\delta u_x}  \tfrac12\partial_x(u_x^2)
e^{-\tfrac{\sigma^2}{2}\int u^2}.   
\end{equation}
Integrating by parts gives,
\begin{equation}
      - \tfrac12 \int du~f\int dx   \frac{\delta }{\delta u_x} \left( 
\partial_x(u^2 )
e^{-\tfrac{\sigma^2}{2}\int u^2}\right).  
\end{equation}
Now $\frac{\delta }{\delta u_x} \left( 
\partial_x(u^2 )
e^{-\tfrac{\sigma^2}{2}\int u^2}\right) = \frac{\delta }{\delta u_x} \left(
\partial_x(u^2 )\right)
e^{-\tfrac{\sigma^2}{2}\int u^2}+ 
\partial_x(u^2 )\frac{\delta }{\delta u_x} 
e^{-\tfrac{\sigma^2}{2}\int u^2}$ and $\frac{\delta }{\delta u_x} 
\partial_x(u^2 )= \frac{\delta }{\delta u_x} 2u 
\partial_x u = 2\partial_xu$ and 
\begin{equation}\label{1.11}
\partial_x (u^2)   \tfrac{\delta }{\delta u_x} 
e^{-\tfrac{\sigma^2}{2}\int u^2}  =  -\sigma^2 u   
\partial_x (u^2) e^{-\tfrac{\sigma^2}{2}\int u^2} = -\sigma^2\partial_x\tfrac23(u^3)
e^{-\tfrac{\sigma^2}{2}\int u^2} .
\end{equation}
Crucially $
    u\partial_x(u^2)  = \partial_x\tfrac23u^3
$ is an exact derivative and hence if $f$ only depends on $u_x$, $x\in [-L,L]$,
\begin{equation}
    \int du f \int dx \partial_x( 2u -\sigma^2 \tfrac23u^3)e^{-\tfrac{\sigma^2}{2}\int u^2} =0. 
\end{equation} 
Note that the Burgers part of the flow formally preserves white noise with {\it any} variance parameter $\sigma^2$.  The constraint $\sigma=1$ is  set by the Edwards-Wilkinson part.

\subsection{Smoothed out version  \cite{MR3350451} } For the carefully smoothed out stochastic Burgers equation 
\begin{equation}
  \label{SoSBE}
    \partial_t u = \tfrac12\partial_x (u^2\ast R) +\tfrac12\partial_x^2 u + \partial_x\eta,   
\end{equation}
where $\eta$ is Gaussian with $\langle\eta(t,x),\eta(s,y)\rangle = \delta(t-s)R(x-y)$  the previous argument becomes rigorous.  Again it is not hard to see that white noise with covariance $R$ is invariant for the linear part
$\partial_t u = \tfrac12\partial_x^2 u + \partial_x\eta$.  For the nonlinear part the generator is 
$\tfrac12\int dx \partial_x (u^2\ast R)(x) \frac{\delta}{\delta u_x} $.  The analogue of \eqref{1.11} (with $\sigma=1$) is
\begin{equation}
   \partial_x (u^2\ast R)(x)  \frac{\delta}{\delta u_x}  e^{-\tfrac12 \langle u, R^{-1} u\rangle} =
  - \partial_x (u^2\ast R)(x) (u\ast R^{-1})(x) e^{-\tfrac12 \langle u, R^{-1} u\rangle},
\end{equation}
and integration with respect to $x$ leads to integration of the exact derivative $\tfrac23 \partial_x u^3$, giving $0$.    Taking $R$ as an approximate identity one can show that the corresponding $Z$ defined by 
$u=\partial_x\log Z$ does indeed converge to the stochastic heat equation \eqref{she} although the proof is quite involved (see  \cite{MR3350451} for details.)  Note that examples of $R$ include the projection $P_N$ to finitely many Fourier modes for the equation on the circle, i.e. on the circle $\partial_t P_N u = P_N \tfrac12\partial_x u^2 $ preserves independent Gaussian Fourier modes.  This fact is used extensively in many proofs, e.g.
\cite{majda2000remarkable,MR2540076}.

\subsection{Discrete version}  The formal argument also holds rigorously with the following special discretization which was used in \cite{MR2570756}, though it goes back at least as far as \cite{PhysRevLett.15.240}:
\begin{equation}
d\phi_x = [\phi_{x+1}-2\phi_x+\phi_{x-1}
+ \phi_{x+1}\phi_x -\phi_x\phi_{x-1}
+ \phi_{x+1}^2-\phi_{x-1}^2 ] dt + dB_{x+1}-dB_x.
\end{equation}
Here $x\in \mathbb{Z}$ and $B_x$ are independent Brownian motions.
The generator $L=S+A$ where 
% $$
% Lf(\phi) = \sum_x ({\partial\over\partial\phi_{x+1}} 
% - {\partial\over\partial\phi_{x}})^2
% + (\Delta \phi  + \phi_{x+1}\phi_x -\phi_x\phi_{x-1}
% + \phi_{x+1}^2-\phi_{x-1}^2) {\partial\over\partial\phi_{x}}
% =S + A
% $$
$S=\sum_x \partial^*_x \partial_x$ is symmetric with respect to $d\mu=\exp\{{- {1\over2} \sum_y \phi_y^2}\}dy$ with the notation $
\partial_x=\frac{\partial}{\partial \phi_x}$
and \begin{equation}
A=
\sum_x( \phi_{x+1}\phi_x -\phi_x\phi_{x-1}
+ \phi_{x+1}^2-\phi_{x-1}^2) {\partial\over\partial\phi_{x}}.
\end{equation}
As before it is clear that
$\int Sf d\mu=0$.  Integrating by parts, 
\begin{align*}
&\int( \phi_{x+1}\phi_x -\phi_x\phi_{x-1}
+ \phi_{x+1}^2-\phi_{x-1}^2) {\partial\over\partial\phi_{x}}f e^{- {1\over2} \sum_y \phi_y^2} \prod d\phi_y
\\&
=-\int f [( \phi_{x+1} -\phi_{x-1}) -\phi_x( \phi_{x+1}\phi_x -\phi_x\phi_{x-1}
+ \phi_{x+1}^2-\phi_{x-1}^2)  ] e^{- {1\over2} \sum_y \phi_y^2} \prod d\phi_y.
\end{align*} The special form of the nonlinearity leads to a telescoping sum over $x$,  since \begin{equation} \phi_x( \phi_{x+1}\phi_x -\phi_x\phi_{x-1}
+ \phi_{x+1}^2-\phi_{x-1}^2)=
\phi_x^2 \phi_{x+1} +\phi_x \phi_{x+1}^2-  \phi_{x-1}^2 \phi_{x} -\phi_{x-1} \phi_{x}^2 \end{equation} is a discrete gradient.  For any local $f$ (depends on finitely many sites) this shows that
$
\int Af d\mu=0$.  In \cite{Gubinelli2017,MR3774431} it is shown that under diffusive scaling the model on $\frac1{N}\mathbb{Z}$ converges to the KPZ equation.

\subsection{WASEP \cite{MR1462228}} Bertini and Giacomin gave the first rigorous proof of Brownian invariance for KPZ 
 employing an approximation by weakly asymmetric simple exclusion process (WASEP).  In the asymmetric simple exclusion process (ASEP) particles on $\mathbb{Z}$ jump in continuous time, to the right at rate $p$ and to the left at rate $q$, with jumps to occupied sites suppressed (exclusion).
The generator has a core of local functions which it acts on as
\begin{equation}\label{asepgenerator}
L f (\eta) = \sum_x \left\{ p\eta(x)(1-\eta(x+1)) +q(1-\eta(x))\eta(x+1)\right\} ( f(\eta^{x,x+1})-f(\eta)),
\end{equation}
where 
$\eta\in \{0,1\}^\mathbb{Z}$ represents the configuration of particles  and $\eta^{x,x+1}$ is obtained from $\eta$ by switching the occupation variables at $x$ and $x+1$.  
For any $\rho\in [0,1]$ the Bernoulli product measure $\pi_\rho$ on $\{0,1\}^\mathbb{Z}$ with $\pi_\rho(\eta(x)=1)=\rho$ and $\pi_\rho(\eta(x)=0)=1-\rho$ is invariant. To see this, let $f$ depend only on $\eta(x)$, $|x|\le N$.   Making the change of variables $\eta\mapsto\eta^{x,x+1}$ 
\begin{equation}
\int \eta(x)(1-\eta(x+1)) f(\eta^{x,x+1})d\pi_\rho (\eta)  = \int \eta(x+1)(1-\eta(x)) f(\eta) d\pi_\rho (\eta),
\end{equation}
since $\frac{d\pi_\rho(\eta^{x,x+1})}{d\pi_\rho(\eta)} =1$.  Hence
\begin{equation}
\int Lfd\pi_\rho = \int (p-q)\sum_x ( \eta(x+1)(1-\eta(x)) - \eta(x)(1-\eta(x+1))) f(\eta)d\pi(\eta).
\end{equation}
Again the summation is telescoping and leads to 
$
\int g fd\pi_\rho 
$
where $g$ does not depend on the variables $\eta(x)$, $|x|\le N$ and is mean $0$.  Hence $\int Lf d\pi_\rho=0$.

The height function $h$ is defined to go up or down by one at integer points depending on whether or not there is a particle there.  We observe it under the 1:2:4 scaling $h_\eps = \eps^{1/2} h(\eps^{-2} t, \eps^{-1} x)$ with weak asymmetry $p=\tfrac12(1-\eps^{1/2})$, $q=\tfrac12(1+\eps^{1/2})$, in   which case  
\begin{equation} dh_\ep = [ \ep^{-3/2}  ({\bf{1}}_\vee - {\bf{1}}_\wedge) -\ep^{-1} ({\bf{1}}_\vee+{\bf{1}}_\wedge)]  dt + d M,\end{equation} where 
$M$ is a martingale and ${\bf{1}}_\vee$ means that we are  at a local min of $h$, ${\bf{1}}_\vee$  that we are  at a local max.  It is not completely transparent why this is a discretization of the KPZ equation.  But on the lattice $\ep \mathbb Z$,
\begin{equation}
{\bf{1}}_\vee - {\bf{1}}_\wedge  =\tfrac{ \ep^{3/2}}{2} \nabla^- \nabla^+ h\,,\qquad 
{\bf{1}}_\vee + {\bf{1}}_\wedge  =-\tfrac{ \ep}{2} \nabla^-h \nabla^+ h + \tfrac12\,,
\end{equation}
where $\nabla^- h(x) = \ep^{-1} (h(x)- h(x-\ep)) $, $\nabla^+ h(x) = \ep^{-1} (h(x+\ep)- h(x)) $, and the martingales are approximating white noises, one can see that formally WASEP converges to the KPZ equation
\eqref{KPZ} modulo the  large drift $\eps^{-1}t/2$. 
It is proved in \cite{MR1462228} by using the fact that the exponential of the height function satisfies a nice discretization of \eqref{she}.

 $\pi_\rho$ are actually the extremals of the set of translation invariant probability measures invariant for ASEP.  But there are other invariant measures which are not translation invariant, e.g. the {\em blocking measures} which are product measures with $\mu(\eta(x)=1) = \frac{(p/q)^x}{1+ (p/q)^x}$.  These turn out not have non-trivial limits in the weakly asymmetric limit:  It has recently been shown \cite{janjigian,Dunlap-Sorensen-2024}, that Brownian motion with drifts are the only invariant measures for the KPZ equation on the line.  That the Brownian bridge is the only invariant measure on the circle was proved earlier \cite{Hairer-Mattingly-2018}.

 \subsection{Generator} \cite{MR4168394} employ martingale problems to construct  the stochastic Burgers equation as a Markov process for initial data absolutely continuous with respect to the invariant measure $\mu$, in this case spatial white noise.  They are able to construct a rich enough domain for the generator $L$ that in a certain sense $L^*\mu=0$ is identifying $\mu$ as invariant.  Their method in fact goes much farther and shows exponential $L^2$-ergodicity.  On the other hand, the domain is so tailor made to the equation that it does not seem possible to use it to prove convergence theorems.  Most relevant to the our discussion is that $L^*\mu=0$ is not proven directly, but, as in all other articles, inherited from special discrete models.

 \subsection{KdV} Invariance of white noise for the KdV equation 
\begin{equation}
    \partial_t u =\tfrac12 \partial_x u^2 + \lambda\partial_x^3 u
\end{equation} is formally related since the invariance by the linear part of the flow $\partial_t u = \lambda\partial_x^3 u$ is clear, \blue{by  the rotational invariance of the complex Gaussian.}
It was first done on the circle using integrable methods \cite{MR2365449}, then with non-integrable methods \cite{MR2540076}, then extended to the line using clever integrable arguments \cite{MR4145790}.

\blue{\subsection{Integration by parts method}The new proof presented in this paper is based on an integration by parts identity which comes from  the polymer representation. The identity itself comes from a mollification procedure, but unlike most earlier proofs of invariance, the mollification here does not have to be chosen very precisely.  The integration by parts formula is  mostly a standard Gaussian integration by parts, but there is a hidden cancellation -- an unusual It\^o formula -- which can be thought of as the mechanism behind the white noise conservation. The  crux of the argument lies in proving \eqref{e.cancellation}, which is the hidden symmetry and, in our telling of the story, is what keeps white noise invariant. The integration by parts formula itself has other applications. For example, in a forthcoming article we use it to study polymer coalescence.}

\blue{We will present the result for the equation posed on $\R$. The periodic setting can be treated similarly, with some simplifications. }

% We now move on to the main steps of the proof.

\blue{\subsection*{Organization of the paper.} The rest of the paper is organized as follows. In Section~\ref{s.result}, we present the main results: two integration by parts formulas, one with respect to the Gaussian noise and the other with respect to the Gaussian initial data. By combining these formulas with Stein’s equation, we establish the invariance of white noise under \eqref{SBE}. The subsequent sections are dedicated to proofs, with Section~\ref{s.cancellation} focusing on the crucial cancellation in \eqref{e.cancellation}.}

\section{Main results} \label{s.result}

\blue{Let $Z(t,x;y)$ denote the solution to \eqref{she} starting from $Z(t=0,x;y)=\delta_y(x)$.  By linearity, the solution with initial data $Z_0(x)$ is given by $\int_\R Z(t,x;y)Z_0(y)dy$.  Letting $Z_0(y)= e^{h_0(y)}$, the \emph{Cole-Hopf solution of KPZ} \eqref{KPZ} starting with initial height function $h_0$ is $h=\log Z$.  As long as $h_0$ grows at most linearly at $\pm \infty$ the integral makes sense and gives the unique solution to \eqref{she} with initial data $e^{h_0}$ (see, for example, \cite{MR1462228}).  This defines a stochastic process $u=\partial_x h $  in the space of distributions, with initial data $u_0=\partial_x h_0$, the \emph{stochastic Burgers flow}. More precisely, for any smooth $f$ with compact support on $\R$ we have a real stochastic process $\langle u_t,f\rangle$. We say $u$ is \emph{a white noise} if $\langle u,f\rangle$ are jointly Gaussian with $E[ \langle u,f\rangle^2]= \|f\|_{L^2(\R)}^2$.
\begin{theorem}  The stochastic Burgers flow preserves white noise:
    If $u_0$ is a white noise, then $u_t$ is a white noise.
\end{theorem}
}

The density (in $y$) at time $t$ of the polymer starting at $x$  with reward $e^{h_0}$ is
\begin{equation}\label{e.qdensity}
\rho_{t,x}^{\scriptscriptstyle\rm poly}(y)=\frac{Z(t,x;y)e^{h_0(y)}}{\int_\R Z(t,x;\tilde{y})e^{h_0(\tilde{y})}d\tilde{y}}.
\end{equation}

Our main tool for proving invariance is the following formula.

\begin{proposition}[Integration by parts for KPZ]\label{p.yt12}Let $\xi$ be space-time white noise,  and $u_t(x)$ the Cole-Hopf solution of \eqref{SBE} with   smooth initial data $u_0$ satisfying $|\int_0^x u_0(y)dy|\leq \alpha+\beta |x|$ for some $\alpha,\beta>0$.  Then for any smooth $f$ with compact support on $\R$ and $F\in C^1(\R)$,
\begin{equation}
\label{e.yt1} \begin{aligned} &\EE F(\langle f,u_t\rangle) \langle f,u_t\rangle= \|f\|_{L^2(\R)}^2\EE F'(\langle f,u_t\rangle) + \Gamma\end{aligned}\end{equation}
where
\begin{equation}
\label{last} \begin{aligned}
& \Gamma= \tfrac12\int_{\R^4}dx_1dx_2dy_1dy_2 f(x_1)f'(x_2){\rm sgn}(y_2-y_1) \EE F'(\langle f,u_t\rangle)\rho_{t,x_1}^{\scriptscriptstyle\rm poly}(y_1)\rho_{t,x_2}^{\scriptscriptstyle\rm poly}(y_2)\\
&\quad\quad+  \EE F(\langle f,u_t\rangle) \int_{\R^2} dxdy ~u_0(y)\rho^{\scriptscriptstyle\rm poly}_{t,x}(y) f(x).  \end{aligned}
\end{equation}
\end{proposition}

Here, and elsewhere, $\EE$ refers to expectation over the random background $\xi$.  The initial data $u_0$ is fixed here, either non-random, or the formula holding for each realization of it.  

Suppose that $u_0$ is a stationary Gaussian process with a smooth covariance \begin{equation}\label{aa}
    {\rm Cov}(u_0(x),u_0(y))= (\psi\ast\psi) (x-y).
\end{equation} 
In other words, $u_0$ is white noise convolved with a smooth symmetric test function $\psi$. 
Define $a(x)$ by
\begin{equation}\label{ab}
    a'(x) = (\psi\ast\psi) (x),\qquad a(0)=0.
\end{equation}
 Then we can take the expectation $\Einit$ over the initial data and perform another Gaussian integration by parts on the second term on the right hand side of  \eqref{last}.  The result is

\begin{proposition}\label{Prop13}  Let $u_0$ be the stationary Gaussian process with covariance \eqref{aa} and $a(\cdot)$ is defined in \eqref{ab}, then 
\begin{equation*}  \begin{aligned}
& \qquad\qquad \Einit F(\langle f,u_t\rangle) \int_{\R^2} dxdy ~u_0(y)\rho^{\scriptscriptstyle\rm poly}_{t,x}(y) f(x)\\
& \quad = -\int_{\R^4}dx_1dx_2dy_1dy_2 f(x_1)f'(x_2)a(y_2-y_1) \Einit F'(\langle f,u_t\rangle)\rho_{t,x_1}^{\scriptscriptstyle\rm poly}(y_1)\rho_{t,x_2}^{\scriptscriptstyle\rm poly}(y_2).  
\end{aligned}
\end{equation*}
    \end{proposition}

Now we can finish the proof of invariance. 
 Note first of all that we can assume without loss of generality that our spatial white noise has mean zero.  This is because if $u$ is a solution of stochastic Burgers \eqref{SBE} and $m$ is a constant, then $v= u+m$ solves $\partial_t v = \tfrac12\partial_x v^2 \blue{-} m\partial_x v + \tfrac12\partial_x^2 v + \partial_x\xi$, so if mean zero spatial white noise is invariant, then mean $m$ spatial white noise is as well. Equivalently, showing Brownian motion is invariant for KPZ up to height shifts implies that Brownian motion with constant drift is as well.

 We start initially with $u_0=\partial_x h_0$ as a   spatial  white noise, run \eqref{KPZ} up to time $t$ and we want to prove that the solution $u_t$ at time $t$ is again a spatial white noise.  Equivalently, for any test function $f\in C^\infty_0(\R)$, $Y=\langle f,u_t\rangle$
has normal distribution with mean $0$ and variance $\sigma^2=\|f\|^2_{L^2(\R)}$.
By Stein's equation, it is enough to show that for any $F\in C^1(\R)$,  
\begin{equation}\label{e.stein}
E F(Y)Y=\sigma^2 E F'(Y),
\end{equation}
where $E$ is the expectation with respect to both the random background $\xi$ and the initial white noise $u_0$.  But \eqref{e.stein} follows from Prop. \ref{p.yt12} and Prop. \ref{Prop13} taking the limit as $\psi\ast\psi$   approximates the identity, i.e., $a(x)\to \tfrac12 {\rm sgn}(x)$.

In other words, if $u_0$ is spatial white noise, $\Einit\Gamma =0$.

\begin{remark}  Take initial data $u_0= $ white noise $+$ a smooth function $\bar{u}_0$.  We can write the solution to \eqref{SBE} as $u(t,x)= u^{\rm eq}(t,x) + \bar{u}(t,x)$ where $u^{\rm eq}$ is the solution of \eqref{SBE} with the white noise initial data (i.e. stochastic Burgers equilibrium solution) and 
$\bar{u}$ solves %\blue{to be corrected}
\begin{equation}
    \partial_t \bar{u} = 
    \partial_x (u^{\rm eq} \bar{u}) + \tfrac12 \partial_x \bar{u}^2 + \tfrac12\partial_x^2 \bar{u},\qquad \bar{u} (0,x) = \bar{u}_0(x)
\end{equation}
so we can write $\bar{u}=\partial_x\log \bar{Z}$ where
\begin{equation}
    \bar{Z}(t,x) = \E^{\scriptscriptstyle\rm poly,eq}_{t,x}[\bar{Z}_0(X_t)].
\end{equation}
Here $\E^{\scriptscriptstyle\rm poly,eq}_{t,x}$ is the expectation with respect to the polymer measure at stationarity.
 So far, this is standard.  On the other hand, if we use Prop. \ref{p.yt12} we get the following identity
 \begin{equation}
\label{e.yt1'} \begin{aligned} &\EE F(\langle f,u_t\rangle) \langle f,u_t-v_t\rangle= \|f\|_{L^2}^2\EE F'(\langle f,u_t\rangle), \qquad v_t(x)= \Epoly_{t,x} \bar{u}_0(X_t).  
\end{aligned}
\end{equation}
\end{remark}

\begin{remark} 
The integration by parts formula \eqref{e.yt1} can be viewed as a variant of the replica method, with only two replicas considered. Essentially, we computed  the two point covariance function $\EE \log Z(t,x_1)\partial_{x_2} \log Z(t,x_2)$. Using the replica trick, this expression can be written as:
\begin{equation}\label{e.replica}
\begin{aligned}
&\EE \log Z(t,x_1) \partial_{x_2}\log Z(t,x_2)=\lim_{n\to0} \EE \tfrac{Z^n(t,x_1)-1}{n} \partial_{x_2} \tfrac{Z^n(t,x_2)-1}{n}\\
&=\lim_{n\to0} \EE \tfrac{Z^n(t,x_1)-1}{n} \tfrac{n Z^{n-1}(t,x_2)\partial_{x_2} Z(t,x_2)}{n}\\
&=\lim_{n\to0} \EE \tfrac{Z^n(t,x_1)-1}{n}Z^{n-1}(t,x_2)\partial_{x_2} Z(t,x_2).
\end{aligned}
\end{equation}
For positive integer $n$,    $\EE Z^n(t,x_1)Z^{n-1}(t,x_2)\partial_{x_2} Z(t,x_2)$ can be written  in terms of $2n$ independent copies of   Brownian motions, and the usual problem is that we can not directly send $n\to0$ in that expression. The key trick here is to avoid integrating out all the Gaussian random variables in the Feynman-Kac representation. Instead, we focus only on the ``interaction" between $\partial_{x_2} Z(t,x_2)$ and $Z^{n-1}(t,x_2)$, $Z^n(t,x_1)$,  while neglecting the ``interaction" within $Z^{n-1}(t,x_2)$ and $Z^n(t,x_1)$, because for small $n$, it is difficult to compute it. On a formal level, integration by parts yields:
\[
\begin{aligned}
\EE \tfrac{Z^n(t,x_1)-1}{n}&Z^{n-1}(t,x_2)\partial_{x_2} Z(t,x_2)=\EE Z^{n-1}(t,x_1) Z^{n-1}(t,x_2) \\
&\times \E e^{\sum_{j=1}^2\int_0^t \xi(t-s,x_j+B^j_s)ds+h_0(x_j+B^j_t)} \int_0^t \delta'(x_2+B^2_s-x_1-B^1_s)ds.
\end{aligned}
\]
This holds for all $n>0$, and by sending $n\to0$ we obtain the same result as directly applying integration by parts to the left-hand side of \eqref{e.replica}:
\[
\begin{aligned}
\EE \log Z(t,x_1)  &\partial_{x_2} \log Z(t,x_2)=\EE Z^{-1}(t,x_1) Z^{-1}(t,x_2)\\
&\times \E e^{\sum_{j=1}^2 \int_0^t \xi(t-s,x_j+B^j_s)ds+h_0(x+B^j_t)} \int_0^t \delta'(x_2+B^2_s-x_1-B^1_s)ds.
\end{aligned}
\]
\end{remark}

\section{Integration by parts}\label{s.ibp}

The previous discussion shows that the non-trivial input is the integration by parts formula presented in Prop.~\ref{p.yt12}.  In order to prove it we will first prove a version with $\xi$ convolved in space by another smooth symmetric test function $\phi$ with finite support such that $\int \phi=1$.  We call
\begin{equation}\label{sown}
    \eta (t,x) =\int_\R dy~
    \phi(x-y)\xi(s,y).
\end{equation}
The smoothed out white noise has covariance \begin{equation}
    \EE \eta(t,x)\eta(s,y) = \delta(t-s) R(x-y)\qquad {\rm where}\qquad R=\phi\ast\phi.
\end{equation}
% \blue{In the following we use $\partial_x \eta$ as the spatial derivative of $\eta$ which is another Gaussian process that is white in time and smooth in space.}
The Feynman-Kac representation of the solution to the stochastic heat equation is 
\begin{equation}\label{fkr}
Z(t,x)=\E e^{\int_0^t \eta(t-s,x+B_s)ds-\frac12R(0)t+h_0(x+B_t)},
\end{equation}
where $\E$ denotes expectation with respect to a standard Brownian motion $B$  on $\R$ starting at the origin.
Through the Cole-Hopf transformation, we can write the solution to the stochastic Burgers equation as
\begin{equation}\label{e.fkueps}
u(t,x)=\frac{ \E e^{\int_0^t \eta(t-s,x+B_s)ds+h_0(x+B_t)}[\int_0^t \partial_x\eta(t-s,x+B_s)ds+u_0(x+B_t)]}{\E e^{\int_0^t \eta(t-s,x+B_s)ds+h_0(x+B_t)}}.
\end{equation}

It is natural to write this as $u(t,x)= \Epoly_{t,x} [\int_0^t ds\partial_x\eta(t-s,X_s)+u_0(X_t)]$ where the expectation $\Epoly_{t,x}$ is with respect to the \emph{polymer measure} $\mathbb{P}^{\scriptscriptstyle\rm poly}_{t,x}$,  the measure on continuous paths $X_\cdot$ on $[0,t]$
starting at $x$ which is the Wiener measure tilted by the factor $e^{-H}$ where 
\begin{equation}
    H= H(t,x, X_\cdot)= -\int_0^t \eta(t-s,X_s)ds-h_0(X_t).
\end{equation} 
In other words, for any bounded functional $\Phi:C_x[0,t]\to\R$,
\begin{equation}
\E^{\scriptscriptstyle\rm poly}_{t,x} \Phi=\frac{\E e^{\int_0^t \eta(t-s,x+B_s)ds+h_0(x+B_t)} \Phi(x+B_\cdot)}{\E e^{\int_0^t \eta(t-s,x+B_s)ds+h_0(x+B_t)}}= Z^{-1}(t,x) \E_{x} \Phi e^{-H},
\end{equation}
where $\E_x$ denotes expectation with respect to Brownian motion starting at $x$.
If we  write $\rho^{\scriptscriptstyle\rm poly}_{t,x}(s,y) $ for the density at time $s$ of the polymer path $X_\cdot$ (\blue{which can actually be written explicitly in terms of the Green's function of SHE}), 
\begin{equation}
\rho^{\scriptscriptstyle\rm poly}_{t,x}(s,y) =\E^{\scriptscriptstyle\rm poly}_{t,x}\delta_{y}(X_s),
\end{equation}
the first term in \eqref{e.fkueps} reads
\begin{equation}
   \int_0^t ds\int_\R dy~ \partial_x\eta(t-s,y) \rho^{\scriptscriptstyle\rm poly}_{t,x}(s,y). 
\end{equation}%\gucomment{The way it's written the stochastic integral in (2.8) is only formal, do we need to say something like ``the stochastic integral here should be interpreted in the Stratonovitch sense, and one should note that $\rhopoly(s,\cdot)$ depends on $\eta$ in $[0,t]$''? This is related to the discussion at the beginning of Section 4}
With some abuse of notations,   the $\rhopoly(y)$ defined in \eqref{e.qdensity} is just $\rhopoly(t,y)$.
The polymer measure depends on the whole environment $\eta$ as well as the initial data $h_0$ and this stochastic integral should be interpreted in the Stratonovich sense. %\adcomment{Is this sufficient?} 

\begin{proposition}\label{p.yt120}
Let $u_t$ be the solution of the stochastic Burgers equation \eqref{SBE} forced by the smoothed out noise $\partial_x\eta$ \eqref{sown} and with initial data $u_0$.  Then with the polymer densities corresponding to the smoothed out noise $\eta$ described above,
\begin{equation}\label{e.yt1smooth}
\begin{aligned}
&\EE F(\langle f,u_t\rangle) \int_0^t ds \int_{\R} dy~ \partial_x\eta(t-s,y) \rho^{\scriptscriptstyle\rm poly}_{t,x}(s,y)   \\&=\int_{\R^3} dx_2 dy_1dy_2f'(x_2) \EE F'(Y)  \int_0^t ds~ R'( y_2-y_1) \rho_{t,x}^{\scriptscriptstyle\rm poly}(s,y_1)\rho_{t,x_2}^{\scriptscriptstyle\rm poly}(s,y_2).
\end{aligned}
\end{equation}
\end{proposition}

The proofs of Prop. \ref{p.yt120}  and Prop. \ref{Prop13} are  rather straightforward integration by parts in the Gaussian spaces associated with $\eta$ and $u_0$. And although we use the language of Malliavin calculus, essentially all we are using is that 
$
\EE F(Z_1)Z_2=\EE F'(Z_1)\mathrm{cov}[Z_1,Z_2]
$ for two correlated centered Gaussian random variables $Z_1,Z_2$. The proofs are given in Section \ref{ibp}. For an introduction to basic Malliavin calculus, we refer to \cite[Chapter 1]{Nua06}.  

At this point, we continue with the proof of Prop. \ref{p.yt12} given Prop. \ref{p.yt120}.
Defining
\begin{equation}
    r'(x)= R(x),\qquad r(0)=0,
\end{equation}
we can write the last term in \eqref{e.yt1} 
\begin{equation}
    \int_0^t ds \int_{\R^2}dy_1dy_2R'( y_2-y_1) \rho_{t,x}^{\scriptscriptstyle\rm poly}(s,y_1)\rho_{t,x_2}^{\scriptscriptstyle\rm poly}(s,y_2)= \Epoly_{t,x_1,x_2} \int_0^t ds~ r''( X^2_{s}-X^1_s).
\end{equation}
Here $\Epoly_{t,x_1,x_2}$ is the quenched expectation on the two independent polymer paths starting at $x_1,x_2$ respectively. \blue{By It\^o's formula (for Brownian motions)}, \begin{equation}
\int_0^t ds~ r''(X^2_s-X^1_s)=r(X^2_t-X^1_t)- r(x_2-x_1) -\int_0^t R(X^2_s-X^1_s)d(X^2_s-X^1_s).
\end{equation}
To proceed further we choose $R$ and $\phi$  to be 
\begin{equation}\label{e.Reps}
     R_\eps(x) = \eps^{-1}R(\eps^{-1} x),\qquad \phi_\eps(x)= \eps^{-1}\phi(\eps^{-1} x).
\end{equation}
As $\eps\to 0$ it is straightforward to check that 
\begin{equation}
r_\eps (x)\to \tfrac12{\rm sgn}(x)
\end{equation}
and the polymer densities converge, so that Prop. \ref{p.yt12} is recovered from Prop. \ref{p.yt120} in the limit once one has the non-trivial fact that 
\begin{equation}\label{e.cancellation}
\lim_{\eps\to 0} \Epoly_{t,x_1,x_2}\int_0^t R_\eps(X^2_s-X^1_s)d(X^2_s-X^1_s) =0. 
\end{equation}This is the hidden symmetry that keeps white noise invariant. More precisely, we define a class of initial conditions satisfying growth conditions, 
\begin{equation}\label{e.Bab}
\mathscr{B}_{\alpha,\beta} 
= \{ h_0: |h_0(x)| \le \alpha + \beta |x|\}.
\end{equation} We will show in the next section that

\begin{proposition}
\label{p22}
    For any $h_0\in \mathscr{B}_{\alpha,\beta}$ there is a $C$ depending on $t$ such ,
    \begin{equation}
    \label{e:315}
        \EE \big|\Epoly_{t,x_1,x_2}\int_0^t R_\eps(X^2_s-X^1_s)d(X^2_s-X^1_s)\big| \le Ce^{C(|x_1|+|x_2|)} \blue{\eps^{1/2} |\log\eps|^{1/4}.}
    \end{equation}
\end{proposition}
The bound on the right hand side suffices to take the desired limits to obtain Prop. \ref{p.yt12} because the test functions $f$ there are assumed to have compact support.

\section{Proof of Proposition~\ref{p22}}
\label{s.cancellation}

\blue{The goal of this section is to establish the key cancellation in \eqref{e.cancellation}. There are different ways to understand this result. Our approach relies on considering the first meeting time of the two paths. Once they meet, they continue from the same location and thus become exchangeable. Since the Itô integral in \eqref{e.cancellation} is antisymmetric, it follows that the contribution to the integral after the meeting time is zero. Furthermore, the fact that 
 $R_\eps$ approximates the Dirac function ensures that the integral before the meeting time vanishes as $\eps\to0$.}
 
 \blue{ An alternative, and potentially more intricate, way to analyze the integral is to rewrite it as 
\[
\int_0^t \int_{\R^2}R_\eps(y_2-y_1)[u(t-s,y_2)-u(t-s,y_1)]\rho^{\scriptscriptstyle\rm poly}_{t,x_2}(s,y_2)\rho^{\scriptscriptstyle\rm poly}_{t,x_2}(s,y_2)dy_1dy_2ds.
\]
On the formal level, as $R_\eps\to \delta$, the term $u(t-s,y_2)-u(t-s,y_1)$ vanishes in the limit. However since the above expression is purely symbolic when $\eps=0$, we do not pursue this approach further, though it is instructive to mention it here.}

 We first provide estimates for general $R(\cdot)$. In the end, $R$ will be chosen as in \eqref{e.Reps} to complete the proof. Let $\mathcal{E}_{t,x_1,x_2}$ denote the term in absolute values in
\eqref{e:315}.
We can write it as
\begin{align}\label{e.4.1}
&\mathcal{E}_{t,x_1,x_2}=
\frac{\E_{x_1,x_2}  e^{\sum_{j=1}^2\int_0^t \eta(t-s,X^j_s)ds-\frac{t}2R(0)+h_0(X^j_t)}\int_\sigma^\tau R(X^1_s-X^2_s)d(X^1_s-X^2_s) }{\E_{x_1,x_2} e^{\sum_{j=1}^2\int_0^t \eta(t-s,X^j_s)ds-\frac{t}2R(0)+h_0(X^j_t)}},
\end{align}
where $\E_{x_1,x_2}$ is the expectation with respect to two independent Brownian motions starting at $x_1,x_2$, and 
\begin{equation}\label{def:sigmatau}
    \sigma= \inf\{ s\ge 0 : X^1_s-X^2_s \in {\rm supp} \,R\} \wedge t,\qquad \tau= \inf\{ s\ge 0 : X^1_s-X^2_s =0\} \wedge t.
\end{equation}
This is because the integration from $0$ to $\sigma$ vanishes because $R=0$ there and the integration from $\tau$ to $t$ vanishes because, if $\tau<t$ then after $\tau$, $X^1$ and $X^2$ are exchangeble, but the integral is antisymmetric.

It is hard to analyse the polymer measure directly.  On the other hand, for fixed $t>0$, the denominator in \eqref{e.4.1} will have a nice strictly positive limit as the smoothing kernel becomes an identity, and it is the top that is small.   This inspires us to estimate it by a simple application of Cauchy-Schwarz which separates the numerator and denominator. Let's call the numerator 
\begin{equation}\label{e.defXeps}
\begin{aligned}
\mathcal{X}_{t,x_1,x_2}=\E_{x_1,x_2}  e^{\sum_{j=1}^2\int_0^t \eta(t-s,X^j_s)ds-\frac{t}2R(0)+h_0(X^j_t)}\int_\sigma^\tau R(X^1_s-X^2_s)d(X^1_s-X^2_s)
\end{aligned}
\end{equation}
and note that the denominator 
\begin{equation}\label{den}
  \E_{x_1,x_2}  e^{\sum_{j=1}^2\int_0^t \eta(t-s,X^j_s)ds-\frac{t}2R(0)+h_0(X^j_t)}= Z(t,x_1)Z(t,x_2),  
\end{equation}
where $Z(0,x) = e^{h_0(x)}$.

\begin{lemma}\label{l.cs} \blue{Assuming $R(\cdot)=\tfrac{1}{\eps}\phi\star\phi(\tfrac{\cdot}{\eps})$}. There exists a $C<\infty$ depending on $t,\alpha,\beta<\infty$ but not $\eps$  such that for any $h_0\in \mathscr{B}_{\alpha,\beta} $,
\begin{equation}
\EE |\mathcal{E}_{t,x_1,x_2}| \leq C e^{ \beta(|x_1|+|x_2|)} \sqrt{\EE \mathcal{X}_{t,x_1,x_2}^2}.\end{equation} 
% \begin{equation} B_t(h_0) = C (q_{t,x_1}(h_0)q_{t,x_2}(h_0))^{-1}
% \qquad q_{t,x}(h_0)= \int dy  p_t(x,y)e^{h_0(y)}.
% \end{equation}
\end{lemma}

\begin{proof}
% \adcomment{ The Hu reference is for wick ordered and I don't know they really proved that it doesn't depend on R.  I don't know if it works for our equation, though perhaps they are the same in this case. Maybe the easiest way out is to have an appendix where we redo Moreno-Flores's proof in this context.\blue{I was a referee of Hu's paper. It's indeed on wick ordered noise which  covers our white in time case (wick = ito here). It draws inspirations from Moreno's proof, but the approximation scheme is different. I can double check again but I'm pretty sure that their argument covers our case (see the proof of their Theorem 4.6). If I remember correctly, Moreno's proof used your AKQ result which is an i.i.d. random variable approximation of the white noise, so it's unclear to me how it deals with $\xi_\eps$?}}
 By Cauchy-Schwarz inequality, 
 \begin{equation}
\EE |\mathcal{E}_{t,x_1,x_2}| \leq (\EE (Z(t,x_1)Z(t,x_2))^{-2})^{1/2}(\EE \mathcal{X}_{t,x_1,x_2}^2)^{1/2}.
\end{equation}
We can write 
\begin{equation}
    Z(t,x) = \int_{\R} dy  \, p_t(x-y)e^{h_0(y)}z(t,x,y),
\end{equation}
where $p_t(x)=(2\pi t)^{-1/2}e^{-\frac{x^2}{2t}}$ is the standard heat kernel and $z(t,x,y)$ can be written in terms of  $\E_{x\to y}$, the expectation over the Brownian bridge from $x$ to $y$:
\begin{equation}
    z(t,x,y) = \E_{x\to y}[ e^{\int_0^t \eta(t-s,X_s)ds-\frac{t}2R(0)}].
\end{equation}
It suffices to bound $\EE Z(t,x)^{-4}$. By Jensen's inequality applied to the function $x^{-4}$ which is convex on $\R_+$, we have 
\[
\EE Z(t,x)^{-4} \leq \int_{\R}  dy \,p_t(x-y)e^{-4h_0(y)} \EE z(t,x,y)^{-4} .
\]
Note that the latter has $\EE z(t,x,y)^{-4}$, which is independent of  $x$ or $y$ by affine invariance of $\xi$, and it has a  finite upper bound independent of $R(\cdot)$ \cite[Corollary 4.8]{HL22}.  
%Therefore the result follows by Jensen's inequality applied to the function $x^{-2}$ which is convex on $\R_+$ with a constant $C (q_{t,x_1}(h_0)q_{t,x_2}(h_0))^{-1}$ where $q_{t,x}(h_0)= \int dy  p_t(x,y)e^{h_0(y)}$.  
Now if 
$h_0\in \mathscr{B}_{\alpha,\beta} $,
we have 
\begin{align*}
     \int_{\R} dy p_t(x-y) e^{-4h_0(y)}\le \int_{\R} dy p_t(x-y) (e^{4\alpha + 4\beta y } + e^{4\alpha - 4\beta y }),
\end{align*}
which gives the required bound.
\end{proof}

\begin{lemma}\label{l.2ndmmre} Let $\E_{\x}$ be expectation with respect to a Brownian motion in $\R^4$ starting at $\x$.  Then
\begin{equation}\label{e.2ndmmX}
\begin{aligned}
\EE \mathcal{X}_{t,x_1,x_2}^2= \E_{x_1,x_2,x_1,x_2}& e^{ \int_0^t \VV({\bf X}_{s})ds+\LL({\bf X}_t)}
%\\&\times
\prod_{j=1}^2 \int_{\sigma_j}^{\tau_j} R(X_{s}^{2j-1}-X_{s}^{2j})d(X_{s}^{2j-1}-X_{s}^{2j}),
\end{aligned}
\end{equation}
where \begin{equation}\label{VVee} \VV(\x)=\sum_{1\le i<j\le 4}R(x_i-x_j)\qquad \LL(\x) = \sum_{i=1}^4 h_0(x_i). 
\end{equation}
Note here that $\sigma_j=\sigma(X^{2j-1}_\cdot-X^{2j}_\cdot)$ and $\tau_j=\tau(X^{2j-1}_\cdot-X^{2j}_\cdot)$ are functionals of the paths $X^{2j-1}-X^{2j}$ defined analogously to \eqref{def:sigmatau} as the first time up to $t$ that they enter the support of $R(\cdot)$, and hit $0$.
\end{lemma}

% Without the second line, the r.h.s. of \eqref{e.2ndmmX} is actually the fourth moment of the solution to the SHE with the mollified noise $\xi_\eps$, which solves the (mollified) Delta Bose gas. 

It is convenient to view 
 $e^{ \int_0^t \VV({\bf X}_{s})ds+\LL({\bf X}_t)}
 $ as the weight of a Gibbs measure on $C([0,t],\R^4)$.  We write expectations with respect to this measure as $\Egibbs_{\x}$. 
\blue{The proof of Lemma~\ref{l.2ndmmre1} below shows}  the path $\bfX_s$ under the Gibbs measure is  a diffusion 
\begin{equation}\label{e.sdebfX}
\begin{aligned}
 &d\mathbf{X}_s=\bfU_{t-s}(\mathbf{X}_s)ds+d\mathbf{B}_s, \quad\quad \bfX_0=\x,\\
 &\mbox{ with } \bfU_t(\x)= (U_t^1(\x),\ldots,U_t^4(\x))=\nabla \log \cZ(t,\x), \mbox{ and } \mathbf{B}_s=(B_s^1,\ldots,B_s^4).
 \end{aligned}
 \end{equation}
 Here $\cZ(t,\x)= \E_{\x} e^{ \int_0^t \VV({\bf X}_{s})ds+\LL({\bf X}_{t})}$ solves 
 $
\partial_t \cZ=\tfrac12\Delta \cZ+\VV(\x) \cZ$ with  $\cZ(0,\x)=e^{\LL(\x)}$.
We can write  
\[
 \EE \mathcal{X}^2_{t,x_1,x_2}=\tfrac{ \EE \mathcal{X}^2_{t,x_1,x_2}}{\cZ(t,x_1,x_2,x_1,x_2)}\cZ(t,x_1,x_2,x_1,x_2)
 \]
 and we will show $\cZ$ is uniformly bounded by $e^{C(|x_1|+|x_2|)}$ independent of $R(\cdot)$ (see \eqref{e.417} below), then it suffices to show that the first factor on the r.h.s. goes to zero.

 \begin{lemma}\label{l.2ndmmre1}
 \begin{align}
\nonumber\tfrac{\EE \mathcal{X}_{t,x_1,x_2}^2}{\cZ(t,x_1,x_2,x_1,x_2)}=\Egibbs_{x_1,x_2,x_1,x_2} &\int_{\sigma_{1}}^{\tau_{1}} R(X_s^1-X_s^2)\blue{\left([U^1_{t-s}(\bfX_s)-U^2_{t-s}(\bfX_s)]ds+dB_s^1-dB_s^2\right)}\\
\label{e.11131} \times& \int_{\sigma_{2}}^{\tau_{2}} R(X_s^3-X_s^4)\blue{\left([U^3_{t-s}(\bfX_s)-U^4_{t-s}(\bfX_s)]ds+dB^3_s-dB^4_s\right).}
 \end{align}
 \end{lemma}
 
 \begin{proof}
 This is a general fact for Gibbs measure coming from the Feynman-Kac representation of the solution to the heat equation with a potential. We present the proof for  convenience of the reader:   
 $\H(t,\x)=\log \cZ(t,\x)$ solves
\[
\partial_t \H=\tfrac12\Delta \H+\tfrac12|\nabla \H|^2+\VV, \quad \H(0,\x)=\LL(\x).
\]
Fix $t>0$. Applying It\^o's formula to $\H(t-s,\mathbf{X}_s)$ with $\mathbf{X}_0=\x$,
\[
\begin{aligned}
\LL(\mathbf{X}_t)=\H(t,\x)-\int_0^t \VV(\mathbf{X}_s)ds+\int_0^t \bfU_{t-s}(\mathbf{X}_s)\cdot d\mathbf{X}_s-\tfrac12\int_0^t |\bfU_{t-s}(\mathbf{X}_s)|^2 ds.
\end{aligned}
\]
Thus, the Gibbs ``density'' can be written as 
\begin{equation}\label{e.rnde}
\tfrac{e^{\int_0^t \VV(\mathbf{X}_s)ds+\LL(\mathbf{X}_t)}}{\cZ(t,\x)}=e^{\int_0^t \bfU(t-s,\mathbf{X}_s)\cdot d\mathbf{X}_s-\tfrac12\int_0^t |\bfU(t-s,\mathbf{X}_s)|^2 ds}
\end{equation}
and the representation follows from Girsanov theorem.
 \end{proof}

 \begin{lemma}\label{l.mmlocaltime}  
 Suppose that $h_0\in \mathscr{B}_{\alpha,\beta}$ \blue{(defined in \eqref{e.Bab})}.  Then there is a $C<\infty$ depending on $T,\alpha,\beta$  such that for all $t\in[0,T]$, $\x\in \R^4$, \begin{equation}\label{e.417}
    \cZ(t,\x)\le C e^{\beta \sum_{i=1}^4|x_i|} 
 \end{equation}  and 
 \begin{equation}
    |\bfU(t,\x)| \leq  Ce^{C\beta \sum_{i=1}^4|x_i|}\big(1+|\log |{\rm supp}\,R||\big).
 \end{equation}
 \end{lemma}
 
 \begin{proof}  First we bound $\cZ(t,\x)$ using the Feynman-Kac representation $\E_{\x} e^{ \int_0^t \VV({\bf X}_{s})ds+\LL({\bf X}_{t})}$.  By Cauchy-Schwarz inequality,
\begin{equation}\label{e.416}
   \cZ(t,\x)\le (\E_\x e^{2\int_0^t \VV({\bf X}_{s})ds})^{1/2}
   (\E_\x e^{2\LL({\bf X}_{t})})^{1/2} \le C{e^{Ct}}\prod_{i=1}^4(\int_\R dy p_t(x_i- y) e^{2h_0(y)})^{1/2},
 \end{equation}
 so that if $h_0\in \mathscr{B}_{\alpha,\beta}$, \eqref{e.417} holds. 
 
To study $\bfU=\nabla \log \cZ$, by symmetry it suffices to consider $\partial_{x_1}$.  By the mild formulation, 
\begin{equation}\label{e.mild}
 \partial_{x_1} \cZ(t,\x)= \partial_{x_1} \pp_t\ast e^{\LL} (\x)+\int_0^tds\int_{\R^4} d\y\partial_{x_1} \pp_{t-s}(\x-\y)\VV(\y)\cZ(s,\y),
 \end{equation}
 where $\pp_t$ is the standard heat kernel on $\R_+\times \R^4$ with $\pp_t(\x) = \prod_{i=1}^4p_t(x_i)$.   The first term on the right hand side, \begin{equation}
 \partial_{x_1} \pp_t\ast e^{\LL} (\x) = \int_\R dy_1~ \partial_xp_t(x_1- y_1) {e^{h_0(y_1)}}\prod_{i=2}^4\int_\R dy_i ~p_t(x_i- y_i) {e^{h_0(y_i)}},
 \end{equation}
 from which it is not hard to see that if $h_0\in \mathscr{B}_{\alpha,\beta}$ there is a $C<\infty$ depending on $t,\alpha$ so that 
\begin{equation}
|\partial_{x_1} \pp_t\ast e^{\LL} (\x) |\le Ce^{\beta\sum_{i=1}^4 
|x_i|}.
\end{equation}
The second term on the right hand side  of \eqref{e.mild} can be bounded by \eqref{e.417} as
% \begin{equation}
% \begin{aligned}
%&\sum_{i<j}\int_0^tds\int_{(\R^4)^2} d\y d\z~ R(y_i-y_j)|\partial_xp_{t-s}(x_1-y_1)| p_{s}(y_1-z_1) e^{h_0(z_1)}\times\\&\qquad \times\prod_{\ell=2}^4 p_{t-s} (x_\ell-y_\ell)p_{s} (y_\ell-z_\ell)e^{h_0(z_\ell)} 
% \end{aligned}
% \end{equation}and the computation depends on whether $i=1$ or not.  If $i=1$ we use 
 \begin{equation}\label{e.4.21}
\begin{aligned}
&C\sum_{i<j}\int_0^tds\int_{\R^4} d\y ~ R(y_i-y_j)|\partial_xp_{s}(x_1-y_1)| \prod_{\ell=2}^4 p_{s} (x_\ell-y_\ell)e^{\beta\sum_{i=1}^4|y_i|},
 \end{aligned}
 \end{equation}
 where we have also changed $t-s\mapsto s$.
 Since $|x|\le e^{x^2/4}$ for $x\in \R$, we have, for any $s>0$, $x\in\R$,
 \begin{equation}
 |\partial_x p_s (x)|\le Cs^{-1/2} p_{2s}(x).
 \end{equation}
 Let $\lambda=|{\rm supp} \,R|$.  There is a constant $C>0$ such that $R(x)\le Cp_{C\lambda}(x)$.
 Since $e^{\beta|y|}\le e^{\beta y} + e^{-\beta y}$ we can bound \eqref{e.4.21} by terms for each $i<j$ which are all essentially the same as
  \begin{equation}\label{e.4.22}
\begin{aligned}
&C\int_0^tds \,s^{-1/2}\int_{\R^4} d\y ~ p_{C\lambda}(y_i-y_j)p_{2s}(x_1-y_1) \prod_{\ell=2}^4 p_{s} (x_\ell-y_\ell)\prod_{i=1}^4 (e^{\beta y_i}+e^{-\beta y_i}).%e^{\pm\beta\sum_{i=1}^4y_i},
 \end{aligned}
 \end{equation}
 %where the $\pm$ indicates we need a bound for both cases.  
 Consider the case $(i,j)=(1,2)$. The integral over $\y\in \R^4$ can be computed exactly with a result bounded by
 \begin{equation}
     C p_{3s+C\lambda}(x_2-x_1) e^{C\beta \sum_{i=1}^4|x_i|}.
 \end{equation}
 At this point, the best that one can do is bound $p_{3s+C\lambda}(x_2-x_1)\le C (3s+C\lambda )^{-1/2}$ and then the integration in $s$ results in the factor $C|\log \lambda|$. 

 To get the lower bound of $\cZ$, first, $\cZ(t,\x)\geq \E_\x e^{\LL({\bf X}_{t})}$ by the fact $R\geq0$. With $h_0$ satisfying  $h_0(x)\geq -\alpha-\beta|x|$,  we further derive $\cZ(t,\x)\geq C^{-1}e^{-\beta\sum_{i=1}^4 |x_i|}$, which completes the proof.
 \end{proof}

 By Schwarz's inequality, symmetry, and Lemma~\ref{l.2ndmmre1},
 \begin{align} 
&\tfrac{\EE \mathcal{X}_{t,x_1,x_2}^2}{\cZ(t,x_1,x_2,x_1,x_2)}\le  \Egibbs_{\x}  (\int_{\sigma_1}^{\tau_1} R(X_s^1-X_s^2)[U^1_{t-s}(\bfX_s)-U^2_{t-s}(\bfX_s)]ds)^2\\
&\blue{+2\Egibbs_{\x}  \int_{\sigma_1}^{\tau_1} R(X_s^1-X_s^2)[U^1_{t-s}(\bfX_s)-U^2_{t-s}(\bfX_s)]ds\int_{\sigma_2}^{\tau_2} R(X_s^3-X_s^4)d(B^3_s-B^4_s)=:I_1+I_2}
%& \le C(1+|\log |{\rm supp }R||)~\Egibbs_{x_1,x_2,x_1,x_2}  (\int_\sigma^\tau R(X_s^1-X_s^2){e^{C\beta \sum_{i=1}^4 |X_s^i|}}ds)^2
\label{e.11131}.
 \end{align}
 \blue{We apply Schwarz's inequality to bound $I_2$ by
 \[
 2\left(I_1\times 2\Egibbs_{\x}\int_{\sigma_2}^{\tau_2} R^2(X_s^3-X_s^4)ds\right)^{1/2}.
 \]}
 \blue{So it suffices to estimate $I_1$.} Multiplying by $\cZ(t,x_1,x_2,x_1,x_2)$ and then applying H\"older's inequality and Lemma~\ref{l.mmlocaltime},
\begin{equation}\label{e.326}
\begin{aligned}
 I_1\le& C(1+|\log |{\rm supp }R||)~(\E e^{ 4\int_0^t \VV({\bf X}_{s})ds})^{1/4}
  (\E e^{4 \LL({\bf X}_t)})^{1/4} \\
  &\times (
\E (\int_{\sigma_1}^{\tau_1} R(X_{s}^{1}-X_{s}^{2})e^{C\beta \sum_{i=1}^4 |X_s^i|}ds )^4 )^{1/2},
\end{aligned}
\end{equation}
where the expectations are wrt $\E=\E_{x_1,x_2,x_1,x_2}$.  By H\"older's inequality again,
\begin{equation}
  \E e^{ 4\int_0^t \VV({\bf X}_{s})ds} \le \E e^{C\int_0^t R(B_s) ds} \le {Ce^{C t}},
\end{equation}
where $B_s$ is a Brownian motion with variance $2$, which for the purposes of an upper bound we may as well assume starts at $0$.  By independence and symmetry,
\begin{equation}
( \E e^{4 \LL({\bf X}_t)})^{1/4} = (\int_{\R} dy p_t(x_1-y) e^{4h_0(y)})^{1/2}(\int_{\R} dy p_t(x_2-y) e^{4h_0(y)})^{1/2}\leq Ce^{2\beta(|x_1|+|x_2|)}.
\end{equation}
For the last factor on the right hand side of \eqref{e.326}, we bound $e^{C\beta \sum_{i=1}^4 |X_s^i|} \leq e^{C\beta \max_{s\in[0,t]}\sum_{i=1}^4 |X_s^i|}$, and for any $p\geq1$,
\[
\E e^{pC\beta \max_{s\in[0,t]}\sum_{i=1}^4 |X_s^i|} \leq C'e^{2pC\beta (|x_1|+|x_2|)}.
\]After applying Cauchy-Schwarz again we are left with the factor $(\E(\int_\sigma^\tau R(X_s^2-X_s^1)  ds)^8)^{1/4}$, \blue{and so 
\[
I_1\leq C(1+|\log |{\rm supp }R||) e^{C\beta(|x_1|+|x_2|)}(\E(\int_\sigma^\tau R(X_s^2-X_s^1)  ds)^8)^{1/4}.
\]This also implies\[
\begin{aligned}I_2\leq &C(1+|\log |{\rm supp }R||)^{1/2} e^{C\beta(|x_1|+|x_2|)}\\
&\times(\E(\int_\sigma^\tau R(X_s^2-X_s^1)  ds)^8)^{1/8}(\E\int_\sigma^\tau R^2(X_s^2-X_s^1)  ds)^{1/2}.
\end{aligned}
\]
}

\blue{Let's put together what we have:  If $h_0\in \mathscr{B}_{\alpha,\beta}$, $\EE |\mathcal{E}_{t,x_1,x_2}| \leq C e^{C\beta(|x_1|+|x_2|)}\sqrt{I_1+I_2}$, where $I_1,I_2$ are bounded as described above. It remains to replace $R$   by $R_\eps(\cdot)=\tfrac{1}{\eps}R(\tfrac{\cdot}{\eps})$ and derive  
 moment estimates of $\int_\sigma^\tau R(X^2_s-X^1_s)ds$. We claim that 
\begin{equation}\label{e:summary}
\E(\int_\sigma^\tau R_\eps(X_s^2-X_s^1)  ds)^8\leq C \eps^8, \mbox{ and }  \E \int_\sigma^\tau R_\eps^2(X_s^2-X_s^1)  ds \leq C.
\end{equation}
% \begin{equation}\label{e:summary}
% \begin{aligned}
%    \EE |\mathcal{E}_{t,x_1,x_2}| \leq C(1+|\log |{\rm supp }R||)^{1/2}e^{C\beta(|x_1|+|x_2|)}   &(
% \E_{x_1,x_2} (\int_\sigma^\tau R(X_{s}^{2}-X_{s}^{1})ds )^8 )^{1/8}\\
% +C(1+|\log |{\rm supp }R||)^{1/4} e^{C\beta(|x_1|+|x_2|)}&(\E_{x_1,x_2}(\int_\sigma^\tau R(X_s^2-X_s^1)  ds)^8)^{1/16} (\E_{x_1,x_2}\int_\sigma^\tau R^2(X_s^2-X_s^1)  ds)^{1/4}.
% \end{aligned}
% \end{equation}
 }
 
% At this point let's take 
% \begin{equation}
%     R(x) = R_\eps(x) = \eps^{-1}R(\eps^{-1} x).
% \end{equation}
To see why \eqref{e:summary} holds, first note $B_s=X_{s}^{2}-X_{s}^{1}$ is a Brownian motion with diffusion coefficient $2$, starting from $x_2-x_1$.  For some $C\in(0,\infty)$, 
$R_\eps(x) \le C\eps^{-1} \mathbf{1}_{[-C\eps,C\eps]}(x)$ and we can only increase the integral by replacing $\sigma$ by the hitting time  of $C\eps$, denoted by $\gamma_{C\eps}$.  By symmetry and the strong Markov property of Brownian motion,
\begin{equation}\label{}
\E_{x_1,x_2} (\int_\sigma^\tau R_\eps(X_{s}^{2}-X_{s}^{1})ds )^8 \le C\E_0[(\eps^{-1}\int_0^{\gamma_{C\eps}} \mathbf{1}_{[0,C\eps]} (B_s) ds)^8],
\end{equation}
where $\E_0$ indicates that $B_0=0$.
By Brownian scaling $\tilde{B}_s = \eps^{-1} B_{\eps^2 s}$ is another Brownian motion.  The latter term becomes
\begin{equation}\label{}
\E_0(\eps^{-1}\int_0^{\gamma_{C_\eps}} \mathbf{1}_{[0,C\eps]} (B_s) ds)^8= \eps^8\E_0(\int_0^{\gamma_C} \mathbf{1}_{[0,C]} (\tilde{B}_s) ds)^8,
\end{equation}
where $\gamma_C$ is the hitting time of $C$.  Let 
$T$ be the hitting time by $\tilde{B}_s$ of $\pm C$.  Each time $\tilde{B}_s$ starts at $0$ it has probability $1/2$ to hit $C$ before $-C$.  Therefore 
\begin{equation}
 \int_0^{\gamma_C} \mathbf{1}_{[0,C]} (\tilde{B}_s) ds \le T_1+\cdots+ T_N, 
\end{equation}
where $T_i$ are independent copies of $T$ and $N$ is an independent geometrically distributed random variable with parameter $1/2$.  Thus it is elementary to bound $\E_0[ (T_1+\cdots + T_N)^8]$.  We conclude that 
\blue{\begin{equation}\label{e:summary1}
    \EE |\mathcal{E}_{t,x_1,x_2}| \leq Ce^{C\beta(|x_1|+|x_2|)} ( \eps |\log \eps|^{1/2} +\eps^{1/2} |\log \eps|^{1/4}),
\end{equation}}
which is \eqref{e:315}.

\section{Gaussian integration by parts: Proofs of Prop.~\ref{p.yt120} and Prop.~\ref{Prop13}}\label{ibp}

We start with the proof of Prop.~\ref{p.yt120}. \blue{Recall that the polymer measure was introduced in Section~\ref{s.ibp}, where we denoted the polymer path by $X_{\cdot}$. It is the Brownian motion weighted by $e^{-H}$ with the Hamiltonian $H$ depending on the random environment $\eta$.} Informally we can write
 \begin{equation}
\int_0^tds~ \partial_x\eta(t-s,X_s)=\int_0^t ds \int_{\R} dy~ \phi'(X_{t-s}-y)\xi(s,y),
\end{equation}
where $\xi(s,y)$ is space-time white noise. It doesn’t quite make sense because the distribution $\xi$ is supposed to act on \emph{deterministic} test functions but $\phi'(X_{t-s}-y)$ depends on the randomness of $\xi$. Alternatively one may try to interpret it as a stochastic integral in the Stratonovitch sense, but we choose not to pursue it here. On the other hand, if $X$ is simply the Brownian motion, then we can treat $\phi'(X_{t-s}-y)$ as a deterministic test function after freezing the Brownian motion, so there is no problem with making sense of the integral. Thus,  what we mean by the expression $ \Epoly_{t,x}\int_0^t ds~\partial_x\eta(t-s, X_s) $ is really 
\[
\tfrac{\E_{x} e^{\int_0^t\eta(t-s,X_s)ds-\frac12R(0)t+h_0(X_t)} \int_0^t ds \int_\R dy \, \phi'(X_{t-s}-y)\xi(s,y)}{Z(t,x)},
\]
with the stochastic integral $\int_0^t ds\int_{\R} dy \, \phi'(X_{t-s}-y)\xi(s,y)$ well-defined for every realization of the Brownian path. With some abuse of notations, we write it as 
\[
\E_x\left[\int_0^t ds \int_{\R} dy   \tfrac{d\Ppoly_{t,x}}{d\Pb} \phi'(X_{t-s}-y)\xi(s,y)\right],
\]
with the Radon-Nikodym derivative 
\[
 \tfrac{d\Ppoly_{t,x}}{d\Pb} =\tfrac{e^{\int_0^t \eta(t-s,X_s)ds-\frac12R(0)t+h_0(X_t)}}{Z(t,x)}.
\]
With this convention and $\D_{s,y}$ representing the Malliavin derivative operator, we can integrate by parts and obtain \begin{eqnarray}
\nonumber&\mathbf{E}\left[F ~\Epoly_{t,x}\int_0^t ds~\partial_x\eta(t-s, X_s) \right]  =\E_x \mathbf{E}   \left[ \int_0^t ds\int_{\R} dy  F \tfrac{d\Ppoly_{t,x}}{d\mathbb{P}}  \phi'(X_{t-s}-y)\xi(s,y) \right] 
\\ \nonumber &\qquad 
=
 \E_x \mathbf{E}\left[ \int_0^t ds\int_{\R} dy~  \cD_{s,y} \left( F \tfrac{d\Ppoly_{t,x}}{d\mathbb{P}}  \right) \phi'(X_{t-s}-y) \right]
\\ \label{lastline}&\qquad 
=  
 \E_x \mathbf{E}\left[\int_0^t ds\int_{\R} dy~ \left( (\cD_{s,y} F )\tfrac{d\Ppoly_{t,x}}{d\mathbb{P}} + F \cD_{s,y}\tfrac{d\Ppoly_{t,x}}{d\mathbb{P}} \right) \phi'(X_{t-s}-y) \right],
\end{eqnarray}
with the last step using the product rule. The use of Fubini is justified by the fact that $F$ is bounded, and the Radon-Nikodym derivative and stochastic integral are in $L^p$ for all $1\le p<\infty$. The next lemma says that the $\cD_{s,y}\tfrac{d\Ppoly_{t,x}}{d\mathbb{P}}$ term does not contribute.

\begin{lemma} $$\int_{\R} dy~ 
 \E_x  \cD_{s,y}\tfrac{d\Ppoly_{t,x}}{d\mathbb{P}}  \phi'(X_{t-s}-y) =0.$$ 
    \end{lemma}

\begin{proof}
    Since $\tfrac{d\Ppoly_{t,x}}{d\mathbb{P}}=\tfrac{e^{-H}}{Z}$, with the Hamiltonian $H=-\int_0^t \eta(t-s,X_s)ds+\frac12R(0)t-h_0(X_t)$ and $Z=Z(t,x)$, 
by the product rule,  \begin{equation}\label{e.product}
 \cD_{s,y}\tfrac{d\Ppoly_{t,x}}{d\mathbb{P}}=   -\left(\tfrac{\cD_{s,y}Z}{Z} +\cD_{s,y}H\right) \tfrac{e^{-H}}{Z} .
\end{equation} Now $\cD_{s,y}H=-\phi(X_{t-s}-y)$ and so 
\[
\int_{\R} dy~ (\cD_{s,y}H)\phi'(X_{t-s}-y)= -\int_{\R} dy~ \phi(X_{t-s}-y)\phi'(X_{t-s}-y)=-R'(0)=0.
\]
Furthermore,
{$
\cD_{s,y} Z =     Z \Epoly_{t,x}\phi(X_{t-s}-y) 
$}
so
\begin{align}\label{fred}
\int_{\R} dy~  \tfrac{\cD_{s,y}Z}{Z} \Epoly_{t,x} \phi'(X_{t-s}-y) & =   \Epoly_{t,x,x} \int_{\R} dy~ \phi(X^2_{t-s}-y)\phi'(X^1_{t-s}-y).
 %  \\
 % & =  \tfrac{e^{-H}}{Z} \E^1_{t,x} \int_0^t ds ~ R'(X^1_{s}-X_{t-s}) 
\end{align}
Note here $X^1_{t-s}$ and $X^2_{t-s}$ are just names of variables under the product polymer measure $\Epoly_{t,x,x}$.
 But  $\int_{\R} dy~ \phi(X^2-y)\phi'(X^1-y)= {R'(X^1-X^2)}$ and $R'$ is an odd function.  The expectation $\Epoly_{t,x,x}$ is symmetric in $X^1,X^2$, so this term vanishes as well. The Fubini in \eqref{fred} is justified since one can put absolute value on $\phi$ and $\phi'$ and the integrations $\int_\R \Epoly_{t,x,x} $ are still finite.
\end{proof}

We conclude that only the first term of \eqref{lastline} contributes.  Now $\cD_{s,y}F(Y) = F'(Y) \cD_{s,y}Y$ and  $Y=-\int f'(x_2)\log Z(t,x_2)dx_2$. By \eqref{fkr}, 
\begin{equation}
\D_{s,y} Y=-\int_{\R}dx_2 f'(x_2)  \tfrac{\cD_{s,y}Z(t,x_2)}{Z(t,x_2)}  =-\int_{\R}dx_2 f'(x_2) {\Epoly_{t,x_2}}\int_0^tds\,\phi(X^2_{t-s}-y). 
\end{equation}
We call the variable $x_2$ to distinguish it from the variable $x$ in \eqref{lastline}.
So
\[
\int_0^tds\int_{\R}dy ~(\cD_{s,y}Y)\phi'(X_{t-s} -y) = - \int_{\R} f'(x_2) \E^{\mathrm{poly},2}_{t,x_2}\int_0^tR'(X_{t-s}-X^2_{t-s})ds dx_2,
\]
where $\E^{\mathrm{poly},2}_{t,x_2}$ means that the expectation is only over the variable $X^2_\cdot$. Now \eqref{e.yt1} follows from \eqref{lastline} after switching $s\mapsto t-s$ and using the antisymmetry of $R'$.

We now turn to the proof of Prop.~\ref{Prop13}, which is rather similar.  Let's write 
\begin{equation}\label{wni}
u_0(X_t)=\int_{\R} dy\, \psi(X_t-y)\zeta(y) dy,
\end{equation}
where $\zeta$ is a spatial white noise.  Note that \eqref{wni} is a mild abuse of notation and we really mean the distribution $\zeta$ is acting on the test function $\psi(X_t-\cdot)$, {when $X_\cdot$ is sampled from the Wiener measure}. Since the test function is in $L^2(\R)$ for any realization of $X_t$, the pairing makes sense and there is no harm in using the formal expression \eqref{wni}.
We can then rewrite  ${\Einit}  F(Y) \Epoly_{t,x} u_0(X_t)$ using the Malliavin derivative $\tilde\cD$  with respect to $\zeta$. We can write the integrand in the first term in Prop.~\ref{Prop13} as
\[
 \Einit  F\tfrac{d\Ppoly_{t,x}}{d\mathbb{P}} \int_{\R} dy \, \psi(X_t-y)\zeta(y)  = \Einit  \int_\R dy ~\tilde\cD_y (F \tfrac{d\Ppoly_{t,x}}{d\mathbb{P}})  \psi(X_t-y).
\]
By the product rule $\tilde\cD_y (F \tfrac{d\Ppoly_{t,x}}{d\mathbb{P}}) = (\tilde\cD_y F) \tfrac{d\Ppoly_{t,x}}{d\mathbb{P}} + F \tilde\cD_y\tfrac{d\Ppoly_{t,x}}{d\mathbb{P}}  $.

\begin{lemma}
{$ \E_x\int_\R dy ~(\tilde\cD_y  \tfrac{d\Ppoly_{t,x}}{d\mathbb{P}} )\psi(X_t-y)=0$.}
\end{lemma}

\begin{proof}
$\tilde\cD_y\tfrac{d\Ppoly_{t,x}}{d\mathbb{P}} = -( \tfrac{\tilde\cD_yZ}{Z} +\tilde\cD_y H)\tfrac{e^{-H}}{Z}$.
Since
$h_0(X_t)=\la \psi\ast 1_{[0,X_t]},\zeta\ra$,
\[\int_\R dy \E_x (  \tfrac{\tilde\cD_yZ}{Z}+\tilde\cD_y H)\tfrac{e^{-H}}{Z} \psi(X_t-y) = {\Epoly_{t,x,x}}\la \psi\ast 1_{[0,X_t^2]},\psi(X_t^1-\cdot)\ra - \la\psi\ast1_{[0,X^1_t]},\psi(X^1_t-\cdot)\ra. \]
If we call
$
 g(w,z)=\la \psi\ast 1_{[0,z]},\psi(w-\cdot)\ra= \int_{-w}^{z-w} A(y)dy$ where $A=\psi\ast\psi$, 
 then we can write the last term as 
 \[ {\Epoly_{t,x,x} } g(X_t^1, X_t^2)   - g(X_t^1,X_t^1). \]
 Now the expectation is symmetric in $X^1$ and $X^2$ so this is the same as $\Epoly_{t,x,x}  \tfrac12(g(X_t^1, X_t^2)+ g(X_t^2,X_t^1))   - g(X_t^1,X_t^1)$.  By symmetry of $A$,
 $g(w,z) = \int_{w-z}^w A(y)dy$. It is easy to check that $g(w,z)+g(z,w)= \int_0^w A(y)dy+\int_0^z A(y)dy$.  Therefore the integrand vanishes. \end{proof} 
 
 Hence we have analogously to the proof of Prop.~\ref{p.yt120}, the only surviving term is 
 \[
  \int_{\R}dx_1 f(x_1) {\Einit}  F'(Y) \int_\R dy~\Epoly_{t,x_1}   \tilde\cD_y Y \psi(X^1_t-y).
 \]
 We call it $x_1$ because the term $\tilde\cD_y Y$ produces a second variable $x_2$.  Specifically,   $\tilde\cD_y Y = - \int_\R dx_2~ f'(x_2) \tfrac{\tilde\cD_y Z(t,x_2)}{Z(t,x_2)} $ with  $\tfrac{\tilde\cD_y Z(t,x_2)}{Z(t,x_2)}= {\Epoly_{t,x_2} }\psi\ast 1_{[0,X_t]}(y) $, 
 which gives
 \[
 -\int_{\R^2}dx_1dx_2 f(x_1)f'(x_2) {\Einit}  F'(Y) {\Epoly_{t,x_1,x_2}} [a(X^2_t-X^1_t) + \int_{-X^1_t}^0A(y)dy ].
 \]
The final term $\Epoly_{t,x_1,x_2} [\int_{-X^1_t}^0A(y)dy ]$ only depends on $x_1$.  Integrating by parts in $x_2$ kills it, giving Prop.~\ref{Prop13}.

 \end{document}